\documentclass[12pt]{amsart}
\usepackage{amsmath,amssymb,amsthm,mathrsfs,enumerate}
\usepackage[all]{xy}
\newcommand{\Res}{\mathrm{Res}}
\newcommand{\Ind}{\mathrm{Ind}}
\newcommand{\cind}{\textrm{c-}\mathrm{Ind}}
\newcommand{\SL}{\mathrm{SL}}

\newcommand{\GL}{\mathrm{GL}}

\newcommand{\Sp}{\mathrm{Sp}}
\newcommand{\spn}{\mathrm{span}}
\newcommand{\inn}{\mathrm{inn}}

\newcommand{\ratF}{F}
\newcommand{\unextF}{L}

\newcommand{\extF}{E}
\newcommand{\D}{D}
\newcommand{\DD}{\mathbb{D}}

\newcommand{\GG}{\mathsf{G}} 
\newcommand{\TT}{\mathsf{T}}

\newcommand{\OF}{\mathcal{O}_{\ratF}}
\newcommand{\OFt}{\mathcal{O}^\times_{\ratF}}
\newcommand{\OD}{{\mathcal{O}_{\D}}}
\newcommand{\ODt}{{\mathcal{O}^\times_{\D}}}
\newcommand{\OL}{\mathcal{O}_{\unextF}}

\newcommand{\eqo}{\equiv_{\ODt}}

\newcommand{\phiE}{\phi_{\extF}}
\newcommand{\PF}{\mathcal{P}_{\ratF}}
\newcommand{\PD}{\mathcal{P}_{\D}}

\newcommand{\Pext}{\mathcal{P}_{\extF}}

\newcommand{\psiE}{\psi_{\extF}}

\newcommand{\T}{\mathbb{T}}
\newcommand{\Stor}{\mathbb{S}}
\newcommand{\resF}{\mathfrak{f}}

\newcommand{\LieT}{\mathfrak{t}}
\newcommand{\LieZ}{\mathfrak{z}}
\newcommand{\LieS}{\mathfrak{s}}
\newcommand{\LieG}{\mathfrak{g}}

\newcommand{\G}{\mathbb{G}}

\newcommand{\ep}{\varepsilon}

\newcommand{\rhop}{\widetilde{\rho}}

\newcommand{\p}{\varpi}

\newcommand{\val}{\mathrm{val}}

\newcommand{\valD}{\mathrm{val}_{\D}}
\newcommand{\valextF}{\mathrm{val}_{\extF}}
\newcommand{\real}{\mathbb{R}}
\newcommand{\nrd}{\mathrm{nrd}}

\newcommand{\apart}{\mathcal{A}}
\newcommand{\buil}{\mathcal{B}}

\newcommand{\tr}{\mathrm{Tr}}
\newcommand{\Hom}{\mathrm{Hom}}

\newcommand{\smat}[1]{\left[ \begin{smallmatrix} #1 \end{smallmatrix} \right]}
\newcommand{\lrc}[1]{\lceil #1 \rceil}

\newcommand{\lconj}[2]{{}^{#1}#2}

\theoremstyle{plain}
\newtheorem{theorem}{Theorem}[section]

\newtheorem{lemma}[theorem]{Lemma}
\newtheorem{proposition}[theorem]{Proposition}
\newtheorem{corollary}[theorem]{Corollary}

\theoremstyle{definition}
\newtheorem{definition}[theorem]{Definition}
\newtheorem{remark}[theorem]{Remark}

\theoremstyle{remark}

\numberwithin{equation}{section}

\begin{document}

\title[Restrictions to the derived group]{Restricting toral supercuspidal representations to the derived group, and applications}
\author{Monica Nevins}
\address{Department of Mathematics and Statistics, University of Ottawa, Ottawa, Canada K1N 6N5}
\email{mnevins@uottawa.ca}
\thanks{This research is supported by a Discovery Grant from NSERC Canada.}
\date{\today}

\begin{abstract}
We determine the decomposition of the restriction of a length-one toral supercuspidal representation of a connected reductive group to the algebraic derived subgroup, in terms of parametrizing data, and show this restriction has multiplicity one. 
As an application, we determine the smooth dual of the unit group of integers $\OD^\times$ of a quaternion algebra $\D$ over a $p$-adic field $\ratF$, for $p\neq 2$, as a consequence of determining the branching rules for the restriction of representations of $\D^\times \supset \OD^\times \supset \D^1$.  
\end{abstract}

\maketitle

\section{Introduction}

Let $\G$ be a connected reductive algebraic group defined over  a local non-archimedean field $\ratF$, and set $G = \G(\ratF)$.  Under certain tameness assumptions, all irreducible supercuspidal representations of $G$ may be constructed in a uniform way, starting from generic cuspidal $G$-data \cite{Adler1998,Yu2001,Kim2007}.  J.~Hakim and F.~Murnaghan \cite{HakimMurnaghan2008} determined the equivalence classes of $G$-data which give rise to isomorphic supercuspidal representations.  
In this paper we consider the subset of \emph{generic toral cuspidal $G$-data of length one} (here abbreviated: $G$-data) and their corresponding supercuspidal representations.  

Let $\G^1$ denote the derived group of $\G$; then this is a connected semisimple group over $\ratF$.  Set $G^1 = \G^1(\ratF)$ and note that this may be strictly larger than the commutator subgroup of $G$.  Restricting a $G$-datum $\Psi$ to $G^1$ produces a datum $\Psi^1$ for $G^1$; in Proposition~\ref{P:generic} we show that $\Psi^1$ is in fact a $G^1$-datum, and that the associated representation $\pi_{G^1}(\Psi^1)$ occurs in the restriction to $G^1$ of the representation $\pi_G(\Psi)$. 
We deduce the full decomposition of the restriction of $\pi_G(\Psi)$ into irreducible (supercuspidal) representations of $G^1$ in Theorem~\ref{T:decomp}.  

In particular the  decomposition of these supercuspidals upon restriction to the derived group has multiplicity one, providing a large class of examples for which \cite[Conjecture 2.6]{AdlerPrasad2006} does hold (although there exist counterexamples to the conjecture in general~\cite{Adlercomm}).  In related work, recently K.~Choiy \cite{Choiy2014} has studied the multiplicities in the restriction to $\SL(n,\mathcal{D})$ of discrete series representations of $\GL(n,\mathcal{D})$, where $\mathcal{D}$ is a central division algebra over $\ratF$.

Our results hold modulo certain hypotheses, which for a tamely ramified group are satisfied when $p$ is sufficiently large.  For example, the simple criterion of genericity that we use here requires $p$ not to be a torsion prime for the dual root datum of $\G$.

It should be possible to generalize these results to $G$-data of length greater than one using the results in \cite{HakimMurnaghan2008} as here.  The most difficult step of the construction, as outlined by J.K.~Yu in \cite{Yu2001} and done in full detail by J.~Hakim and F.~Murnaghan in \cite{HakimMurnaghan2008}, is the consistent choice of Heisenberg-Weil lift; this is a key step for the branching rules as well.  On the other hand, the case of $G$-data of length zero reduces to the case of depth-zero representations, and hence to the analogous question of branching rules for cuspidal representations of Lie groups of finite type.

As an application, we consider the group $\D^\times$, for $\D$ a quaternion algebra over a local non-archimedean field $\ratF$ of odd residual characteristic. The group of $F$-points of its algebraic derived group, which coincides with its commutator subgroup, is $\D^1$, the subgroup of elements of reduced norm $1$.  The representation theory of $\D^\times$ is well-known, having been determined by L.~Corwin and R.~Howe in \cite{Corwin1974, CorwinHowe1977, Howe1971}; that of 
$\D^1$ is described in \cite{Misaghian2005} for example.  We give the branching rules for the restriction of representations of $\D^\times$ to $\D^1$ in Section~\ref{S:branching}.

What is more interesting is the representation theory of the maximal compact open subgroup $\ODt$ of $\D^\times$, which coincides with the group of invertible elements of the integer ring of $D$.  This is not a $p$-adic group, and as such, the methods of the classification of \cite{Adler1998,Yu2001} do not apply.  It is an open problem to classify representations of such groups, which are algebraic groups over local rings.

Using in part the branching rules for $\D^\times$ to $\D^1$ established above, we determine the full representation theory of $\ODt$.  Furthermore, in Section~\ref{SS:class} we prove a parametrization of these representations by equivalence classes of $\ODt$-data, in analogy with the classification for the $p$-adic groups $\D^\times$ and $\D^1$.  

This paper is organized as follows.  We set our notation and recall the notion of genericity for positive-depth quasi-characters of tori in Section~\ref{S:gen}, where we relate these notions for $G$ and $G^1$.   We discuss a key ingredient of the construction, the Heisenberg-Weil lift, in Section~\ref{S:heisenberg}, following \cite{HakimMurnaghan2008}, and prove Proposition~\ref{P:commutes}, which is essential to relating different $G^1$-data in later sections.  

In Section~\ref{S:Yu} we recall the construction of toral supercuspidal representations of length one, following \cite{HakimMurnaghan2008,Yu2001}.   We prove some additional properties of this parametrization in Lemma~\ref{L:equiv1} and Proposition~\ref{P:equiv2}.  Section~\ref{S:derived} gives the branching rules for the restriction of toral supercuspidal representations of length one of $G$ to $G^1$, where the main result is Theorem~\ref{T:decomp}.

We then turn to the case of $G=\D^\times$.  We recall known facts about $G$ and $G^1$, including their representation theory, and prove some needed technical results in Section~\ref{S:division}.  In Section~\ref{S:branching} we apply the preceding to determine the branching rules of the restriction of representations of $\D^\times$ to each of $\D^1$ and $\ODt$.  We use these results to give a classification of the irreducible representations of $\ODt$, up to equivalence, in Theorem~\ref{T:classificationOD}.

An original motivation for considering the branching rules for the pair $(\D^\times, \ODt)$ was to compare them to those for the pair $(\GL(2,\ratF),\GL(2,\OF))$ via the Jacquet-Langlands matching theorem.  We conclude with some remarks on this point in Section~\ref{S:matching}.

\section{Notation and Genericity} \label{S:gen}

Let $\ratF$ be a local nonarchimedean field with residue field $\resF$ and residual characteristic $p$, with integer ring $\OF$, prime ideal $\PF$ with uniformizer $\p$, and valuation function $\val$.   We fix a character $\psi$ of $\ratF$ which is trivial on $\PF$ but nontrivial on $\OF$.  When $\extF$ is an extension field of $\ratF$ then $\valextF$ is normalized to coincide with $\val$ on $\ratF$, and we also choose an extension $\psiE$ of $\psi$ to $\extF$, trivial on the prime ideal $\Pext$. Let $\mu_n\subset \mathbb{C}^\times$ denote the group of $n$th roots of unity. 

Let $\G$ be a connected reductive group defined and tamely ramified over $\ratF$.  Denote by $\G^1 = [\G,\G]$ its derived group and set $G=\G(\ratF)$, $G^1=\G^1(\ratF)$.  Let $Z$ denote the center of $G$.  

We assume that $p$ is sufficiently large for: the existence of generic elements in the Lie algebra ($p$ must not be bad for $G$ \cite[\S 7]{Yu2001}), the decomposition of the Lie algebra of $G$ in the proof of Proposition~\ref{P:generic} ($p > k(\G)$, the order of the kernel of the central isogeny $Z(\G)\times \G^1 \to \G$); the work with the Heisenberg-Weil lift ($p>2$ \cite[\S 2.3]{HakimMurnaghan2008}); and the construction of positive-depth toral supercuspidal representations to apply ($\G$ split over a tamely ramified extension of $F$).  We refer the reader to the excellent discussion in \cite[\S 1]{AdlerRoche2000}.  For the case $G=\D^\times$ considered starting in Section~\ref{S:division}, $p>2$ suffices.

Let $\buil(\G,\ratF)$ denote the (enlarged) Bruhat-Tits building of $\G$ over $\ratF$; then the reduced building $\buil^{red}(\G,\ratF)$ is identified with $\buil(\G^1,\ratF)=\buil^{red}(\G^1,\ratF)$.
To each $x \in \buil^{red}(\G,\ratF)$ and $r \in \real_{\geq 0}$ we associate the corresponding Moy-Prasad filtration subgroups $G_{x,r}$ and $G_{x,r+}$ as in \cite{MoyPrasad1994}.  When $\T$ is a tamely ramified maximal torus of $\G$, these give well-defined filtrations $T_r$ of $T=\T(\ratF)$ and of its Lie algebra.  For any extension field $\extF$ over which $\T$ is split and any $x$ in the apartment $\apart(\G,\T,\extF)$ of $\buil^{red}(\G,\extF)$ corresponding to $\T$, we have filtrations $\G_\alpha(\extF)_{x,r}$ of each root subgroup $\G_\alpha(\extF)$ of $\G(\extF)$ corresponding to $(\G,\T)$.  There are corresponding filtrations, for $r\in \real$, of the Lie algebra and of its dual.  We refer the reader to \cite[\S 2.5]{HakimMurnaghan2008}, for example, for a summary of the many useful properties of these filtrations.

Recall that the \emph{depth} of a representation $\rho$ of $G$ is defined to be the least $r\in \real_{\geq 0}$ such that for some $x\in \buil^{red}(\G,\ratF)$, $\rho$ contains vectors invariant under $G_{x,r+}$.

Let $T$ be a maximal torus of $G$ with Lie algebra $\LieT = \LieT(\ratF)$.
For each $r>0$ we have an isomorphism $e \colon \LieT_r/\LieT_{r+} \to T_r/T_{r+}$ and any character of $\LieT_r/\LieT_{r+}$ is given by $X \mapsto \psi(\langle X^\ast, X \rangle)$ for some $X^\ast \in \LieT^\ast_{-r}$.  

Choose an extension field $\extF$ of $\ratF$ over which $\T$ splits, and let $\Phi = \Phi(\G,\T,\extF)$ be the corresponding root system.  For each $\alpha \in \Phi$, the coroot $\alpha^\vee \colon \G_m \to \T$ is defined over $\extF$ and has linearization at $1$ the element $H_\alpha = d\alpha^\vee(1)\in \LieT(\extF)_0$.  Thus for any $X^\ast \in \LieT^\ast_{-r}$, one has  $\valextF(\langle X^\ast, H_\alpha \rangle) \geq -r$. 

\begin{definition}
An element $X^\ast \in \LieT^\ast_{-r}$ is \emph{$\G(\ratF)$-generic of depth $-r$} if for each $\alpha \in \Phi$,  $\valextF(\langle X^\ast, H_\alpha \rangle) = -r$.
\end{definition}

This definition is taken from \cite[\S 8]{Yu2001}.  Genericity is closely related to the notion of a \emph{good element} of the Lie algebra, defined in \cite{Adler1998}; in this sense the following observation is the analogue of \cite[Lemma 5.9]{AdlerRoche2000}.

Let $T=\T(\ratF)$ be a maximal torus of $G$ and set $\Stor = \T\cap \G^1$.  Then $S=\Stor(\ratF)=T\cap G^1$ is a maximal torus of $G^1$  and for each $r\geq 0$ $S_r=T_r\cap G^1$. Let $Z$ be the center of $G$.  Denote their Lie algebras over $\ratF$ by the corresponding letters $\LieG, \LieG^1, \LieT, \LieS, \LieZ$.  By \cite[Proposition 3.1]{AdlerRoche2000} we have $\LieT = \LieZ \oplus \LieS$ and $\LieS_r = \LieT_r \cap \LieG^1$ for all $r\in\real$.  We may identify $\LieS^\ast$ with the set of $X^\ast \in \LieT^\ast$ which are trivial on $\LieZ$, and reciprocally for $\LieZ^\ast$;  then we have a $T$-invariant decomposition $\LieT^\ast = \LieZ^\ast \oplus \LieS^\ast$.  We may thus uniquely write an element $X^\ast \in \LieT^\ast_{-r}$ as $Z^\ast+Y^\ast$, with $Z^\ast \in \LieZ^\ast_{-r}$ and $Y^\ast \in \LieS^\ast_{-r}$.
Now let $\extF$ be a splitting field of $\T$ (or of $\Stor$) and $\Phi = \Phi(\G,\T,\extF)=\Phi(\G^1,\Stor,\extF)$.  We have $\spn_{\extF}\{H_\alpha \mid \alpha \in \Phi\} = \LieS(\extF)$.  We observe as a consequence that $X^\ast$ is $G$-generic of depth $-r$ if and only if $Y^\ast$ is $G^1$-generic of depth $-r$.

A character $\phi$ of $T$ of positive depth $r$ factors to a representation of $T_r/T_{r+} \cong \LieT_r/\LieT_{r+}$, where it is realized as
$$
\phi(e(X)) = \psi(\langle X^\ast, X \rangle)
$$
for some $X^\ast \in \LieT^\ast_{-r}$; we say $\phi$ is \emph{realized by} $X^\ast$.  Evidently many characters are realized by the same $X^\ast$.  The character $\phi$ is called \emph{$G$-generic of depth $r$} if $X^\ast$ is $G$-generic of depth $-r$.  

\begin{proposition} \label{P:generic}
Let $\T$ be a maximal torus of $\G$ and $T=\T(\ratF)$.  Then a character $\phi$ of $T$ is $G$-generic of depth $r$ if and only if its restriction to $T\cap G^1$ is $G^1$-generic of depth $r$.
\end{proposition}

\begin{proof}
Suppose $\phi$ is a character of $T$ of depth $r$ and let $X^\ast \in \LieT^\ast_{-r}$ realize $\phi$ on $T_r$.  Decompose $X^\ast = Z^\ast + Y^\ast$ with $Z\in \LieZ^\ast_{-r}$ and $Y^\ast \in \LieS^\ast_{-r}$; then $\Res_{S_r}\phi$ is realized by $Y^\ast$ since this decomposition is orthogonal.  The result follows from the observation above.
\end{proof}

\section{On Heisenberg $p$-groups and Weil representations} \label{S:heisenberg}

We summarize some essential components in the construction of supercuspidal representations from \cite[\S 2.3]{HakimMurnaghan2008}.

Let $(W,\langle,\rangle)$ be a finite-dimensional symplectic vector space over $\mathbb{F}_p$.  Endow the set $W \times \mathbb{F}_p$ with the group operation $(w,z)(w',z') = (w+w',z+z'+\frac12\langle w, w' \rangle)$, and denote the resulting Heisenberg group $W^\sharp$. 

 For any choice of nontrivial central character, there is a unique corresponding irreducible representation $\tau$ of $W^\sharp$ by the Stone--von Neumann Theroem.  Because $W^\sharp$ carries a natural action of $\Sp(W)$, $\tau$ extends to a representation $\widehat{\tau}=(\tau_S,\tau)$ of the group $\Sp(W)\ltimes W^\sharp$, called the Heisenberg-Weil lift of $\tau$ \cite[Definition 2.17]{HakimMurnaghan2008}.  This extension is unique in all but one case (which occurs only if $p=3$); in that case, a particular extension has been designated in \cite[\S 2.4]{HakimMurnaghan2008}, attached to the choice of central character of $\tau$.

An abstract $p$-Heisenberg group is a group $H$ which is isomorphic to some $W^\sharp$.  Given any such isomorphism, 
its restriction to the center $Z$ of $H$ induces a map $\mu \colon Z \to \mathbb{F}_p$.  Fixing an isomorphism
$$
\kappa \colon \mu_p\subset \mathbb{C}^\times \to \mathbb{F}_p
$$
allows us to factor $\mu$ uniquely as $\kappa \circ \phi$, for some nontrivial character $\phi$ of $Z$.  In this way $\phi$ alone determines the induced symplectic structure on $H/Z$, which is given on $h,h'\in H$ by $\langle hZ, h'Z\rangle =\kappa(\phi([h,h']))$. 

Therefore conversely, given such a pair $(H,\phi)$, $(H/Z)^\sharp$ is a Heisenberg group and there exist (many) isomorphisms $\nu \colon H \to (H/Z)^\sharp$.  Following \cite[Definition 2.29, Remark 2.33]{HakimMurnaghan2008} we say the isomorphism $\nu$ is \emph{special} if it takes the form $\nu(h) = (hZ,\mu(h))$ and the map $\mu \colon H\to \mathbb{F}_p$ restricts to the character $\kappa \circ \phi$ on $Z$.  It follows that any two special isomorphisms differ by at most an $\mathbb{F}_p$-valued character of $H/Z$.  By \cite[Lemma 2.35]{HakimMurnaghan2008}, any split polarization of $H$ induces a special isomorphism.  

Let $\Sp(H)$ denote the group of automorphisms of $H$ which act by the identity on $Z$.  Any isomorphism $\nu \colon H \to W^\sharp$ induces an isomorphism $\nu_\ast \colon \Sp(H)\to \Sp(W^\sharp)$.  The natural inclusion $\Sp(W)\to \Sp(W^\sharp)$ thereby induces an action of $\Sp(W)$ on $H$ depending on $\nu$.  This allows us to construct the semi-direct product $\Sp(W)\ltimes_\nu H$, which is a group isomorphic to $\Sp(W)\ltimes W^\sharp$ via $1\times \nu$.

Now let $T$ be a group equipped with a homomorphism $f\colon T \to \Sp(H)$.  

\begin{definition}(\cite[Definition 3.17]{HakimMurnaghan2008})
The isomorphism $\nu \colon H \to W^\sharp$ is \emph{relevant} for $f$ 
if the image of the map $\nu_\ast \circ f \colon T \to \Sp(W^\sharp)$ lies in the subgroup $\Sp(W)$.
In this case we write $f_\nu$ for the induced homomorphism $f_\nu \colon T \to \Sp(W)$. 
\end{definition}

In other words, $\nu$ is relevant for $f$ if and only if $f$ induces a group homorphism $f_\nu \times 1 \colon T\ltimes H \to \Sp(W) \ltimes_\nu H$.

Now let $\nu \colon H \to (H/Z)^\sharp$ be a special isomorphism corresponding to the central character $\phi$ and relevant for $f \colon T \to \Sp(H)$.  Let $\tau$ be an irreducible representation of the Heisenberg group $H$ with central character $\phi$.  Then via $\nu$ there is a well-defined Heisenberg-Weil lift of $\tau$ to a representation $\widehat{\tau} = (\tau_S,\tau)$ of $\Sp(H/Z)\ltimes_\nu H$.  Pulling this map back via $f_\nu \times 1$ yields  a representation $\omega$ of $T\ltimes H$, given by
$$
\omega(t,h) = \tau_S(f_\nu(t))\tau(h)
$$
for all $t\in T$, $h\in H$. By \cite[Lemma 3.21]{HakimMurnaghan2008}, the isomorphism class of $\omega$ depends only on the choices of $\phi$ and $f$ and not on $\nu$.

\begin{proposition} \label{P:commutes}
For $i\in \{1,2\}$ let $H_i$ be a Heisenberg group with center $Z_{i}$.  Fix a nontrivial character $\phi_i$ of $Z_i$ and a corresponding special isomorphism $\nu_i \colon H_i \to W_i^\sharp$ where $W_i=H_i/Z_i$.  Let $T_i$ be a group and suppose further that $\nu_i$ is relevant for a homomorphism $f_i \colon T_i \to Sp(H_i)$.
Suppose we have a group isomorphism $\alpha \colon H_1 \to H_2$, inducing $\overline{\alpha}\colon W_1 \to W_2$,  and a group homomorphism $\delta \colon T_1 \to T_2$ such that the following diagrams commute:
$$
\xymatrix{
H_1 \ar@{->}^{\nu_1}[r]\ar@{->}^{\alpha}[d] & W_1^\sharp\ar@{->}^{\overline{\alpha}\times id}[d]  && T_1  \ar@{->}^{(f_1)_{\nu_1}}[rr]\ar@{->}^{\delta}[d] && \Sp(W_1)\ar@{->}^{\inn(\overline{\alpha})}[d]   \\
 H_2 \ar@{->}^{\nu_2}[r]  & W_2^\sharp && T_2  \ar@{->}^{(f_2)_{\nu_2}}[rr] && \Sp(W_2).}
$$
Then we have $\phi_2 \circ \alpha = \phi_1$ on $Z_1$ and $\overline{\alpha}$ is a symplectic isomorphism, whence the maps in the diagram above are well-defined homomorphisms.
Let $(\tau_2,V)$ be a Heisenberg representation of $H_2$ with central character $\phi_2$.  Then $\tau_1 = \tau_2 \circ \alpha$ is a Heisenberg representation of $H_1$ on $V$ with central character $\phi_1$.  Let $\omega_i$ denote the pullback of the Heisenberg-Weil lift corresponding to $\nu_i$ of $\tau_i$ to $T_i \ltimes H_i$.  Then for all $t\in T_1, h\in H_1$ we have
$$
\omega_2(\delta(t),\alpha(h))=\omega_1(t,h),
$$
that is, the lifts coincide.
\end{proposition}

\begin{proof} Since $\nu_i$ is special, restricting the first diagram to $Z_1\subset H_1$ yields $\phi_1=\phi_2\circ \alpha$, and it follows that the symplectic forms on $W_1$ and $W_2$ agree via $\overline{\alpha}$.  The commutativity of the first diagram of the hypothesis ensures that the induced map
$$
\inn(\overline{\alpha})\times \alpha \colon \Sp(W_1)\ltimes_{\nu_1}H_1 \to \Sp(W_2)\ltimes_{\nu_2}H_2
$$
is a homomorphism.  Since $\nu_i$ is relevant for $f_i$, the commutativity of the second diagram of the hypothesis implies that the square in the following diagram commutes:
$$
\xymatrix{
T_1 \ltimes H_1 \ar@{->}^{(f_1)_{\nu_1}\times id}[rrr]\ar@{->}^{\delta \times \alpha}[d]
&&& \Sp(W_1)\ltimes_{\nu_1} H_1 \ar@{->}^{\widehat{\tau_1}}[drr]\ar@{->}^{\inn(\overline{\alpha}) \times \alpha}[d] &&
 \\
T_2 \ltimes H_2 \ar@{->}^{(f_2)_{\nu_2}\times id}[rrr]
&&& \Sp(W_2)\ltimes_{\nu_2} H_2  \ar@{->}^{\widehat{\tau_2}}[rr]
&& GL(V).}
$$
When the Heisenberg-Weil extension is unique, the commutativity of the triangle is immediate.  When instead $p=3$, we note that the compatibility of the choices of extensions of $\tau_i$ to $\Sp(W_i)\ltimes_{\nu_i}H_i$ follows from their explicit dependence on the (compatible) central characters.  The result follows.
\end{proof}

\section{Toral supercuspidal representations of length one}\label{S:Yu}

We summarize the construction of irreducible supercuspidal representations of positive depth arising from toral generic $G$-data of length one.  We assume throughout that the maximal torus $\T$ is tamely ramified over $\ratF$.  We follow the presentation in \cite{HakimMurnaghan2008}.  

If $p$ is sufficiently large to ensure all maximal tori of $\G$ are tamely ramified, then these representations may exhaust the set of supercuspidal representations.  For example, when  $G$ is of rank one (over a separable closure), all irreducible supercuspidal representations of $G$ of positive depth arise either in this way, or else as a twist by a positive-depth character of $G$ of a depth-zero representation.  More generally, this is true of any connected reductive group whose longest tamely ramified twisted Levi sequence (in the sense of \cite[\S 2]{Yu2001}) has two factors, such as $\GL_n$, for $n$ a prime.

\subsection{The datum}   A \emph{generic toral $G$-datum of length one} (abbreviated: $G$-datum) consists of: $T = \T(\ratF$), where $\T$ is a (tamely ramified) minisotropic maximal torus of $\G$, defined over $\ratF$;  a point $y \in \buil^{red}(\G,\ratF) \cap \apart(\G,\T,\extF)$, where $\extF$ is a splitting field of $T$; a $G$-generic quasi-character $\phi$ of $T$ of positive depth $r$; and a quasi-character $\chi$ of $G$ which is either trivial or else of depth $\tilde{r} \geq r$.

\begin{remark}
In \cite{Yu2001}, the construction depends on the choice of $y$ in the enlarged building, but by \cite[Remark 3.10]{HakimMurnaghan2008} we deduce that in the toral case it depends only on the image of $y$ in $\buil^{red}(\G,\ratF)$, which in turn is uniquely determined by the minisotropic torus $T$.
\end{remark}

We abbreviate such a datum as $\Psi = (T,y,\phi,r,\chi)$.  For $g\in G$ we set $\lconj{g}{\Psi}=  (\lconj{g}{T},g\cdot y, \lconj{g}{\phi},r,\chi)$, where $\lconj{g}{T} := gTg^{-1}$ and $\lconj{g}{\phi}$ is the corresponding representation of $\lconj{g}{T}$.  Here and throughout we apply the convention that $s=r/2$.

\subsection{The construction of $\rhop$ } \label{S:Yuconstruction}
The main step is the construction of a representation $\rhop$ of $TG_{y,s}$ from the subset $(T,y,\phi,r)$ of the $G$-datum.  We summarize it here, primarily following the detailed presentation in \cite[\S 2.3 and 3.3]{HakimMurnaghan2008}.

Let $\extF$ be a splitting field of $T$ and set $\Phi = \Phi(\G,\T,\extF)$.  We consider $y$ as an element of $\buil^{red}(\G,\extF)$. 
Define $$J(\extF) = \langle \T(\extF)_r, \G_\alpha(\extF)_{y,s} \mid \alpha \in \Phi\rangle$$ and $$J_+(\extF) = \langle \T(\extF)_r, \G_\alpha(\extF)_{y,s+}\mid \alpha \in \Phi\rangle.$$  
Note that $\T(\extF)J(\extF)=\T(\extF)\G(\extF)_{y,s}$ and $\T(\extF)J_+(\extF)=\T(\extF)\G(\extF)_{y,s+}$.  

The character $\phi$ of $T$ is realized on $T_r/T_{r+} \cong \LieT_r/\LieT_{r+}$ by an element $X^\ast \in \LieT_{-r}^\ast$, via the fixed additive character $\psi$ of $\ratF$.  Extension of scalars from $\resF$ to the residue field of $\extF$ produces from $X^\ast$ a unique linear functional on $\LieT(\extF)_r/\LieT(\extF)_{r+}$ which via $\psiE \circ e$ similarly defines a character $\phiE$ of $\T(\extF)_r/\T(\extF)_{r+}$.  The restriction of $\phiE$ to $T_r$ coincides with $\phi$. 
We may extend $\phiE$ trivially across the groups $\G_\alpha(\extF)_{y,s+}$, $\alpha \in \Phi$,  to produce the character $\widehat\phiE$ of $J_+(\extF)$.

If $J(\extF)=J_+(\extF)$ then let $J=J(\extF)\cap G$ and $\widehat\phi = \Res_{J}\widehat{\phiE}$.  Then it follows that $\phi$ and $\widehat\phi$ together extend to a unique character of $TJ=TG_{y,s}$, which we denote  $\rhop$.  
Note that in this case, one can equally define $\rhop$ without passing to the splitting field \cite[\S 4]{Yu2001}.  We evidently have $\Res_T(\rhop) = \phi$.

Now suppose that  $J(\extF)\neq J_+(\extF)$.  Set $N(\extF)=\ker(\widehat\phiE)$.  The index of $N(\extF)$ in $J_+(\extF)$ is $p$, since this quotient is isomorphic to an additive group of characteristic $p$; in fact $J_+(\extF)/N(\extF) \cong \T(\extF)_r/\ker(\phiE)$.
One verifies that $H(\extF)=J(\extF)/N(\extF)$ is an abstract Heisenberg group over $\mathbb{F}_p$ with center $Z_{H(\extF)} = J_+(\extF)/N(\extF)$.  Since $\phi$ is generic, the construction of $\phiE$ above ensures that the form $\langle hZ_{H(\extF)},h'Z_{H(\extF)}\rangle = \kappa(\phiE([h,h']))$ is nondegenerate and we may set $W(\extF)= (J(\extF)/J_+(\extF),\langle,\rangle)$, a symplectic vector space over $\mathbb{F}_p$.

Since as symplectic vector spaces we have $W(\extF) \cong \oplus_{\alpha \in \Phi}\LieG_\alpha(\extF)_{y,s}/\LieG_\alpha(\extF)_{y,s+}$, we may choose a polarization $W(+)\oplus W(-)$ of $W(\extF)$ where $W(\pm)$ is spanned by the positive (respectively, negative) root spaces.  These Lagrangian subspaces $W(\pm)$ lift to subgroups of $H(\extF)$, thus providing a well-defined splitting of $H(\extF)$.  This implies that each
 $g\in H(\extF)$ may be factored uniquely as $g=g_+g_-g_0$ with $g_{\pm}\in W(\pm)$ and $g_0 \in Z_{H(\extF)}$ and thus the map
$$
\mu(g) = \kappa(\widehat{\phiE}(g_0))+\frac12\langle g_+,g_-\rangle
$$
 defines a special isomorphism $\nu_\extF\colon H(\extF)\to W(\extF)^\sharp$ given by $\nu_\extF(h) = (hZ_{H(\extF)},\mu(h))$ \cite{HakimMurnaghan2008}.

Finally, let $J,J_+$ and $N$ denote the intersections with $G$ of the corresponding groups over $\extF$.  Set $H=J/N$, $W=J/J_+$ and $Z_H = J_+/N \cong Z_{H(\extF)}$.   Since $T$ is minisotropic, it acts by conjugation on $J$, preserving $J_+$ and $N$.  By \cite[Lemma 2.32]{HakimMurnaghan2008}, the restriction $\nu$ of $\nu_\extF$ to $H$ is a special isomorphism with $W^\sharp$, relevant for the map $f'\colon T \to \Sp(H)$ induced by conjugation, and independent of the choice of extension $\psiE$.  By \cite[Lemma 3.18]{HakimMurnaghan2008}, the induced homomorphism $f_\nu \colon T \to \Sp(W)$ coincides with the conjugation action $f$ of $T$ on $W=J/J_+$.

Let $\tau$ denote a Heisenberg representation of $H$ with central character $\phi$, and let $\widehat{\tau} = (\tau_S,\tau)$ denote its Heisenberg-Weil lift to $\Sp(W)\ltimes_\nu H$.  Then the pullback representation of $T \ltimes J$ is given on  $t\in T$ and $j\in J$ by
$$
\omega(t,j)= \tau_S(f(t))\tau(j). 
$$
Set $\rhop(tj) = \phi(t)\omega(t,j)$; this is well-defined and is the representation of $TJ=TG_{y,s}$ we sought.
Note that by \cite[Theorem 11.5]{Yu2001}, $\Res_{T_{0+}}\rhop$ is $\phi$-isotypic; but in general this is not true of $\Res_{T_{0}}\rhop$ due to the presence of the term $\tau_S(f(t))$.

\subsection{The representation $\pi_G(\Psi)$}

Let $\Psi = (T,y,\phi,r,\chi)$ be a $G$-datum.  Construct the representation $\rhop$ of $TG_{y,s}$ from the subset $(T,y,\phi,r)$ as above.   Then
$$
\rho_G(\Psi) = \chi \rhop
$$
is a representation of $TG_{y,s}$.
 The following is a special case of results in \cite{Adler1998,Yu2001}.

\begin{theorem}
The representation
$$
\pi_G(\Psi) = \cind_{TG_{y,s}}^G \rho_G(\Psi)
$$
is an irreducible supercuspidal representation (of depth $r$ if $\chi$ is trivial, else of depth equal to that of $\chi$).
\end{theorem}

We omit the subscript $G$ on $\pi$ and $\rho$ where there is no possibility of confusion.

\subsection{Properties of the parametrization}

J.~Hakim and F.~Murnaghan \cite{HakimMurnaghan2008} determined when two $G$-data give rise to equivalent supercuspidal representations, modulo a hypothesis called $C(\vec{G})$, which is satisfied in the toral case.  We summarize their results for the particular $G$-data we consider here.

\begin{proposition}[Hakim-Murnaghan] \label{P:HM}
Let $\Psi = (T,y,\phi,r,\chi)$ and $\Psi' = (T',y',\phi',r',\chi')$ be two (toral, length-one, generic) $G$-data.
Then
\begin{enumerate}
\item If $T=T'$, $r=r'$ and $\chi\phi = \chi'\phi'$, then $\rho(\Psi) \cong \rho(\Psi')$.
\item We have $\pi(\Psi) \cong \pi(\Psi')$ if and only if there exists $g\in G$ such that $T' = \lconj{g}{T}$, $r=r'$ and $\chi'\phi'=\lconj{g}{(\chi\phi)}$ as characters of $T'$.
\end{enumerate}
\end{proposition}

The first statement is an example of refactorization, and thus follows from \cite[Proposition 4.24]{HakimMurnaghan2008}.  The second, incorporating $G$-conjugacy, is \cite[Corollary 6.10]{HakimMurnaghan2008}.  The proofs of these results involve a detailed and complex analysis of the construction of $\rho$ vis-\`a-vis defined notions of elementary transformations, refactorization and $G$-conjugacy.

Note that from the first statement one may also deduce that if $\chi$ is a character of $G$ of depth less than $r$, then
$\rhop(T,y,(\Res_T\chi)\phi,r) = \chi\rhop(T,y,\phi,r)$.  By our choice of definition of $G$-datum, we always incorporate a twist by a central character into the toral character when its depth is smaller; this shows there is no loss of generality in doing so.

We note the following additional properties of the construction.

\begin{lemma} \label{L:equiv1}
Let $\Psi = (T,y,\phi,r,\chi)$ be a $G$-datum.
\begin{enumerate}[(a)]
\item The center $Z$ acts by the character $\chi\phi$ in $\rho(\Psi)$, and in $\pi(\Psi)$. 
\item If $\Psi'=(T,y,\phi',r,\chi')$ is another $G$-datum such that $\Res_{T_{r}}\phi = \Res_{T_{r}}\phi'$ then the corresponding pullbacks of the Heisenberg-Weil representation are the same, that is, $\omega = \omega'$.
\item If $\Psi'=(T,y, \phi',r,\chi')$ is another $G$-datum such that $\Res_{T_{0}}\phi = \Res_{T_{0}}\phi'$ then $\Res_{T_0G_{y,s}}\rhop(\Psi) =  \Res_{T_0G_{y,s}}\rhop(\Psi')$. 
\item If $\gamma\in G$ then $\lconj{\gamma}{\Psi}$ is a $G$-datum and $\lconj{\gamma}{\rho(\Psi)} \cong \rho(\lconj{\gamma}{\Psi})$.
\end{enumerate}
\end{lemma}

\begin{proof}
We adopt the notation of Section~\ref{S:Yuconstruction}.  To see that the restriction of  $\rho(\Psi)$ to $Z\subseteq T$ is $\chi\phi$-isotypic is
immediate from the construction if $TG_{y,s}=TG_{y,s+}$.  Otherwise, since the conjugation action of $Z$ on $J$ and hence on $W$ is trivial, 
$\tau_S\circ f$ is trivial on $Z$.  Part (a) follows.
For part (b) we assume $TG_{y,s}\neq TG_{y,s+}$.  Note that if $\phi$ and $\phi'$ are characters of depth $r$ coinciding on $T_r$, then they restrict to the same character of $T_r/T_{r+}$.  Thus part (b) is the observation that the dependence of $\omega$ on $\phi$ is limited to the restriction of $\phi$ to this quotient.
Part (c) follows immediately from part (b), and the definition of $\rhop$.

Part (d) is implicit in \cite{HakimMurnaghan2008}.  That $\lconj{\gamma}{\Psi}$ is a $G$-datum is immediate.  For the rest, it suffices to show that the corresponding pullbacks of the Heisenberg-Weil representations to $\lconj{\gamma}T\ltimes \lconj{\gamma}{J}$ coincide, which we do here using Proposition~\ref{P:commutes}. 

We use a subscript $\gamma$ to denote an object in the construction corresponding to the datum $\lconj{\gamma}{\Psi}$.  Since $J_\gamma(\extF)=\lconj{\gamma}{J(\extF)}$ and ${J_\gamma}_+(\extF)=\lconj{\gamma}{J_+(\extF)}$, the result is immediate if $J(\extF)=J_+(\extF)$.

So suppose  $J(\extF)\neq J_+(\extF)$.  The character $\widehat{(\phiE)_\gamma}$ of $\lconj{\gamma}{J_+(\extF)}$ coincides with $\lconj{\gamma}{\widehat{\phi_E}}$, whose kernel is $N(\extF)_\gamma = \lconj{\gamma}{N(\extF)}$.  Similarly, we have $H(\extF)_\gamma = \lconj{\gamma}{H(\extF)}$ and $W(\extF)_\gamma= \lconj{\gamma}{W(\extF)}$.  Moreover, the symplectic form on $W(\extF)_\gamma$ is given by 
$$
\langle \lconj{\gamma}{x},\lconj{\gamma}{y} \rangle_\gamma = \kappa(\lconj{\gamma}{\widehat{\phi_E}}([\lconj{\gamma}{x},\lconj{\gamma}{y}]))=\kappa(\widehat{\phi_E}([x,y])) = \langle x,y \rangle.
$$
It follows that the polarization used in the construction of $\nu_\gamma$ is $W(\pm)_\gamma \cong \lconj{\gamma}{W(\pm)}$.  
Thus for any $h\in H(\extF)$ we have $\mu_\gamma(\lconj{\gamma}{h}) = \mu(h)$, yielding
$$
(\nu_\extF)_\gamma(\lconj{\gamma}{h})=
(\lconj{\gamma}{h}Z_{\lconj{\gamma}{H(\extF)}}, \mu_\gamma(\lconj{\gamma}{h})) 
= (\lconj{\gamma}{(hZ_{H(\extF)})},\mu(h)). 
$$
Descending now to $\ratF$, this implies that conjugation by $\gamma$ gives isomorphisms $\alpha \colon H \to H_\gamma$ and $\overline{\alpha}\colon W\to W_\gamma$ such that the first of the following diagrams
$$
\xymatrix{
H \ar@{->}^{\nu}[r]\ar@{->}^{\alpha}[d] & W^\sharp\ar@{->}^{\overline{\alpha}\times id}[d]  && T  \ar@{->}^{f}[rr]\ar@{->}^{\delta}[d] && \Sp(W)\ar@{->}^{\inn(\overline{\alpha})}[d]   \\
 H_\gamma \ar@{->}^{\nu_\gamma}[r]  & W_\gamma^\sharp && \lconj{\gamma}{T}  \ar@{->}^{f_{\gamma}}[rr] && \Sp(W_\gamma)}
$$
commutes.  Next, letting $\delta \colon T \to \lconj{\gamma}{T}$ denote the conjugation map, we see directly that the second diagram commutes, all maps being the expected conjugations.  

Thus Proposition~\ref{P:commutes} applies.  Let $(\tau_\gamma,V)$ be a Heisenberg representation of $H_\gamma$ with central character $\lconj{\gamma}{\phi}$; then $\tau=\tau_\gamma\circ \alpha$ is a Heisenberg representation of $H$ with central character $\phi$.  Let $\omega_\gamma$ and $\omega$ denote the pullbacks of the Heisenberg-Weil lifts of $\tau_\gamma$ and $\tau$, respectively.  We conclude that for all $t\in T$ and $h\in H$, $\omega_\gamma(\lconj{\gamma}{t},\lconj{\gamma}{h}) = \omega(t,h)$, whence $\omega_\gamma = \lconj{\gamma}{\omega}$.  Recalling that any other choice of $\tau$ with the given central character gives a representation isomorphic to $\omega$, the result follows.
\end{proof}

\begin{proposition} \label{P:equiv2}
Let $\Psi = (T,y,\phi,r,\chi)$ be a $G$-datum.  
Set $S=G^1\cap T$, $\phi^1 = \Res_{G^1}\phi$ and $\chi^1 = \Res_{G^1}\chi$.  
Define
$$
\Psi^1 = \begin{cases}
(S,y,\phantom{\chi^1}\phi^1,r,\chi^1) & \textrm{if the depth of $\chi^1$ is at least $r$, and}\\
(S,y,\chi^1\phi^1,r,1) & \textrm{otherwise}.
\end{cases}
$$
Then $\Psi^1$ is a generic toral length-one $G^1$-datum, called the \emph{restriction of $\Psi$ to $G^1$},  and $\Res_{SG^1_{y,s}}\rho_G(\Psi) \cong \rho_{G^1}(\Psi^1)$.
\end{proposition}

\begin{proof}
First note that $\Stor = \T \cap \G^1$ is a minisotropic maximal torus of $\G^1$ 
associated to the same point $y$ of $\buil^{red}(\G^1,\ratF)=\buil^{red}(\G,\ratF)$.  Setting $S=\Stor(\ratF)$, the character $\phi^1 := \Res_S\phi$ is also $G^1$-generic of depth $r$, by Proposition~\ref{P:generic}.   If $\chi^1=\Res_{G^1}\chi$ has depth less than $r$ then $\Res_{S_r}\chi^1\phi^1=\Res_{S_r}\phi^1$, so this character is also generic of depth $r$.  Thus in each case $\Psi^1$ is a (toral, generic, length-one) $G^1$-datum.

By the remarks following Proposition~\ref{P:HM}, it suffices to prove that $\Res_{SG^1_{y,s}}\rho_G(\Psi) \cong \rho_{G^1}(\Psi^1)$ when $\chi=1$.

Let $\extF$ be a splitting field of $T$ and $S$, and denote the groups arising in the construction for $G^1$ with the superscript $1$.
Since for each root $\alpha$, $\G_\alpha(\extF)_{y,s}=\G^1_{\alpha}(\extF)_{y,s}$, the groups $J^1(\extF)$ and $J^1_+(\extF)$ are defined as for $\G$ but with $\T(\extF)_r$ replaced by $\Stor(\extF)_r$.  It follows that $J^1 \subseteq J$ and $\Res_{J^1_+}\widehat{\phi}=\widehat{\phi^1}$.   Since $J^1=J_+^1$ if and only if  $J=J_+$, the result follows directly in this case.

So suppose $J \neq J_+$.  Since $J_+ \cap J^1 = J_+^1$ and $J^1J_+=J$,  the inclusion $\iota \colon J^1 \to J$ induces an isomorphism $\beta \colon W^1 \to W$.  Since $\Res_{J_+^1}\widehat\phi = \widehat{\phi^1}$, the symplectic forms on  $W^1$ and $W$ coincide under $\beta$, so $\beta$ is a symplectic isomorphism.  Moreover, since $\ker\widehat{\phi^1}=\ker\widehat\phi\cap J^1_+$, we have $N^1=N\cap J^1_+=N\cap J^1$, whence $\iota$ induces also an isomorphism $\alpha \colon H^1\to H$ such that $\beta = \overline{\alpha}$.   Finally, since $\nu$ and $\nu^1$ arise from the same polarization of $W \cong W^1$, we deduce that the diagram
$$
\xymatrix{
H^1 \ar@{->}^{\nu^1}[r]\ar@{->}^{\alpha}[d] & {W^1}^\sharp\ar@{->}^{\beta\times id}[d]  \\
 H \ar@{->}^{\nu}[r]  & W^\sharp}
$$
commutes.  The conjugation action of $S$ on $W^1$ and on $W$ being the same, the second hypothesis of Proposition~\ref{P:commutes} is also satisfied, whence we deduce as before that $\Res_{S\ltimes J^1}\omega \cong \omega^1$, and the result follows.
\end{proof}

\section{On restrictions of representations of $G$ to $G^1$} \label{S:derived}

By \cite{Silberger1979}, the restriction of any irreducible representation of $G$ to $ZG^1$, and hence to $G^1$, decomposes as a finite direct sum of irreducible representations.

Let $\Psi= (T,y,\phi,r,\chi)$ and let $\Psi^1$ denote the restriction of $\Psi$ to $G^1$.  We omit the subscripts $G$ and $G^1$ from the representations.  From Proposition~\ref{P:equiv2} we deduce the irreducible representation $\pi(\Psi^1)$ occurs in $\Res_{G^1}\pi(\Psi)$.  Since $G^1$ is normal in $G$, the remaining summands each have the form $\lconj{\gamma}{\pi(\Psi^1)}$, for some $\gamma \in G$.  
On the other hand by Proposition~\ref{P:HM}, $\pi(\Psi) \cong \pi(\lconj{\gamma}{\Psi})$, so it follows that $\pi((\lconj{\gamma}{\Psi})^1)$
also occurs as a summand of $\Res_{G^1}\pi(\Psi)$.

\begin{lemma} \label{L:mackeygen}
Let $\Psi^1$ be the restriction to $G^1$ of a $G$-datum $\Psi$.
Then for each $\gamma \in G$ we have $(\lconj{\gamma}{\Psi})^1=\lconj{\gamma}{(\Psi^1)}$ and
$$
\lconj{\gamma}{\pi(\Psi^1)} \cong  \pi(\lconj{\gamma}{\Psi^1}).
$$
\end{lemma}

\begin{proof}
Let $\Psi = (T,y,\phi,r,\chi)$ be a $G$-datum; we may assume without loss of generality that $\chi=1$.  Let $\gamma\in G$, which normalizes $G^1$, and write $\Psi^1 = (S,y,\phi^1,r)$.
As $\lconj{\gamma}{T}\cap G^1 = \lconj{\gamma}{S}$, we have $\Res_{\lconj{\gamma}{S}}\lconj{\gamma}{\phi} = \lconj{\gamma}{\phi^1}$.  Therefore the restriction of $\lconj{\gamma}{\Psi}=(\lconj{\gamma}{T},\gamma\cdot y, \lconj{\gamma}{\phi}, r, \chi)$ to $G^1$ coincides with the twisted datum $\lconj{\gamma}{\Psi^1} = (\lconj{\gamma}{S},\gamma \cdot y,\lconj{\gamma}{\phi^1},r)$.
By Lemma~\ref{L:equiv1}, $ \lconj{\gamma}{\rho(\Psi)} \cong \rho(\lconj{\gamma}{\Psi})$, so by Proposition~\ref{P:equiv2}, restricting to $G^1$ yields
$\lconj{\gamma}{\rho(\Psi^1)} \cong \rho(\lconj{\gamma}{\Psi^1})$.
It now follows that
$$
\lconj{\gamma}{\pi(\Psi^1)} = \phantom{\big |}^\gamma\left(\cind_{SG^1_{y,s}}^{G^1}\rho(\Psi^1)\right) \cong \cind_{(\lconj{\gamma}{S})G^1_{\gamma\cdot y,s}}^{G^1}\lconj{\gamma}{\rho(\Psi^1)} = \pi(\lconj{\gamma}{\Psi^1}).
$$
\end{proof}

\begin{theorem} \label{T:decomp}
Let $\Psi$ be a $G$-datum and let $\Psi^1$ denote its restriction to $G^1$.  Then $\pi(\Psi)$ decomposes with multiplicity one upon restriction to $G^1$ as
$$
\Res_{G^1} \pi(\Psi) \cong \bigoplus_{\gamma \in G/TG^1} \pi(\lconj{\gamma}{\Psi^1}).
$$
\end{theorem}

\begin{proof}
By Lemma~\ref{L:mackeygen} and the remarks preceding it, we may apply Mackey theory to deduce that
\begin{align*}
\Res_{G^1}\pi(\Psi) &= \Res_{G^1}\cind_{TG_{y,s}}^{G} \rho(\Psi)\\
&\cong \bigoplus_{\gamma \in G^1\backslash G/TG_{y,s}} \cind_{G^1 \cap \lconj{\gamma}{(TG_{y,s})}}^{G^1} \lconj{\gamma}{\rho(\Psi)}\\
&\cong \bigoplus_{\gamma \in G^1\backslash G/TG_{y,s}}\pi(\lconj{\gamma}{\Psi^1}).
\end{align*}
As $G^1$ is normal in $G$, and for $s>0$, $TG_{y,s} = TG^1_{y,s}\subseteq TG^1$, the given decomposition follows.  

To conclude that the summands are distinct, let $\Psi^1=(S,y,\phi^1,r)$ be the restriction of $\Psi=(T,y,\phi,r,\chi)$ and suppose $\gamma\in G$ is such that $\pi(\lconj{\gamma}{\Psi^1})\cong \pi(\Psi^1)$.  By Proposition~\ref{P:HM} there is some $u\in G^1$ such that setting $g=u\gamma$ we have $\lconj{g}{S}=S$, $g\cdot y=y$ and $\lconj{g}{\phi^1}=\phi^1$.  Since $r>0$, we have $T_r = Z_rS_r$; thus $\lconj{g}{T_r}=T_r$ and $\lconj{g}{J_+} = J_+$.  It also follows that $\lconj{g}{\phi}=\phi$ on $T_r$, whence their trivial extensions to $J_+$ coincide; call this character $\widehat{\phi}$.

A key step in the proof of the irreducibility of $\pi_G(\Psi)$ (\cite[Theorem 9.4]{Yu2001}, also called property $SC1_0$) is the assertion that any $g\in G$ intertwining $\widehat{\phi}$ must lie in $JTJ$.  Thus we conclude that $g \in TJ = TG_{y,s}$ and so $\gamma \in TG^1$, whence the summands are distinct.
\end{proof}

\section{Application to the multiplicative group of the quaternion algebra over $\ratF$} \label{S:division}

Let $\D$ be the quaternionic division algebra over $\ratF$.  We give a self-contained summary of the groups  $\D^\times$ and $\D^1$ and recast their (well-known) representation theory in the language of the preceding sections.  A key reference is \cite[\S 53, 54]{BushnellHenniart2006}.  We assume $p>2$; this satisfies all the hypotheses in Section~\ref{S:gen}.  

\subsection{Notation and background on $\D^\times$}
Let $\ep$ denote a nonsquare in $\OF^\times$. 
Then the quaternion algebra $\D = \DD(\ratF)$ over $\ratF$ can be realized as the $\ratF$-algebra with presentation
$$
\langle 1,i,j,k \mid i^2=\ep, j^2=\p, k^2=-\ep \p, ij=k=-ji \rangle.
$$
Given $z=a+bi+cj+dk$ in this presentation, the anti-involution $z \mapsto \overline{z}=a-bi-cj-dk$ defines the (reduced) trace as $\tr(z)=2a$ and the (reduced) norm as $\nrd(z)=a^2-b^2\ep - c^2\p + d^2\ep \p$, both taking values in $\ratF$. 
The ring $\OD = \{z\in \D \mid \nrd(z) \in \OF\}$ is a maximal compact open subring with unique maximal ideal $\PD=\OD j$.  We normalize our valuation in $\ratF$ so that $\val(\p)=1$ and extend it to a valuation on $\D$ or any algebraic extension field of $\ratF$.  In particular note that $\valD(j)=\frac12$.  

The map $\nrd$ is algebraic over $\ratF$, and the derived group of $\DD^\times$ is $\DD^1 = \ker(\nrd)$.  The groups $\D^\times = \DD^\times(\ratF)$ and $\D^1 = \DD^1(\ratF) \subseteq \OD^\times$ are both compact mod centre.  The Lie algebra of $\D^\times$ is $D$ whereas that of $\D^1$ consists of elements of trace zero.  
One has $[\D^\times,\D^\times] = \D^1$ \cite[Lemma I.4.1]{JacquetLanglands1970} and $[\D^1,\D^1] = \D^1 \cap (1+\PD)$ \cite[\S5]{Riehm1970}. 
The center of $\D^\times$ is $Z= \ratF^\times$; via the norm map $Z\D^1$ has index equal to $\vert \ratF^\times/{\ratF^\times}^2 \vert = 4$ in $\D^\times$.

Each quadratic extension $\extF$ of $\ratF$ can be embedded in $\D$, uniquely up to $\D^\times$-conjugacy, and the restriction of the anti-involution $\overline{\cdot}$ to $\extF$ coincides with the action of the nontrivial Galois element.  Furthermore, for each such $\extF$ there is some $\sigma \in \D^\times$ such that $\lconj{\sigma}{z}=\overline{z}$ for all $z\in \extF$; then $\D = \extF \oplus \sigma \extF$.  Note that $\extF^1 := \extF^\times \cap \D^1$ is given by $\{\beta\overline{\beta}^{-1}\mid \beta \in \extF\}$.  The maximal tori of $\D^\times$ are exactly the groups $\extF^\times$, for $\extF$ a quadratic extension of $\ratF$; there are thus three conjugacy classes.   For each maximal torus $T$ of $\D^\times$, it follows from the norm map that $TD^1$ has index $2$ in $\D^\times$ and that the normalizer in $\D^\times$ of $T$ is $N_{\D^\times}(T) = T \sqcup T\sigma$.  

One may choose explicit representatives as follows.  Denote by $\unextF$ the unramified extension field $\ratF[i]$ contained in $\D$; then one may take $\sigma = j$. 
Fix $\mu \in \unextF^\times$ satisfying $\nrd(\mu)=\ep$. 
Then the two nonconjugate ramified extensions of $\ratF$ in $\D$ are represented by $\ratF[j]$ and $\ratF[\mu j]$; in these cases one may take $\sigma=i$.  

\begin{lemma}\label{L:tori}
There are three conjugacy classes of maximal tori of $\D^1$ when $-1\in (\ratF^\times)^2$.  Otherwise, for each ramified torus $T$ of $\D^\times$, the tori $\D^1 \cap T$ and $\D^1 \cap \lconj{\mu}{T}$ are not $D^1$-conjugate, and there are a total of five $\D^1$-conjugacy classes.
\end{lemma}

\begin{proof}
For each maximal torus $S$ of $\D^1$ there is a maximal torus $T$ of $\D^\times$ such that $S=T \cap \D^1$.  For fixed $T$, the set of $\D^1$-conjugacy classes of tori in $\{\lconj{\gamma}{T}\cap \D^1 \mid \gamma \in \D^\times\}$ is parametrized by $\gamma \in \D^\times/N_{\D^\times}(T)\D^1$.  This group is nontrivial if and only if $T$ is ramified and $-1\notin {\ratF^\times}^2$, in which case it has order two and a set of representatives is $\{1,\mu\}$.  
\end{proof}

One deduces that all maximal tori in $\D^1$ are self-normalizing.

\subsection{Genericity of quasi-characters of tori}

The homomorphism $\varphi \colon\D^\times \to \GL_2(\unextF)$ determined  by  $\varphi(i) = \smat{i & 0\\ 0 & -i}$, $\varphi(j) = \smat{0&1 \\ \p & 0}$ is an embedding.  Its image in $\GL_2(\unextF)$ is the set of fixed points under the involution $\Theta(g) = \varphi(j)^{-1}\overline{g}\varphi(j)$.  Thus we can realize the reduced building  $\buil^{red}(\DD,\ratF)$ of $\D^\times$ as the unique fixed point $x$ in $\buil^{red}(\GL_2,\unextF)$ of the automorphism $\Theta$.  For this choice of $\varphi$, the diagonal split torus is $\Theta$-stable, and $x$ lies in the corresponding apartment $\apart \subset \buil^{red}(\GL_2,\unextF)$, where it 
is the barycentre of the fundamental chamber.  We can and do omit the subscript $x$ from our notation in this case.
For $\GG \in \{\D^\times, \D^1\}$ the Moy-Prasad filtration  subgroups are simply given by $\GG_{r} = \{g \in \GG \mid \valD(g) \geq r\}$ and $\GG_{r+} = \{g \in \GG \mid \valD(g) > r\}$.    
Note that $\D^\times_0 = \ODt$ and $\D^1_0 = \D^1$.

We note in passing that the notion of a generic character of a maximal torus of $\D^\times$ coincides with the original notion of an admissible character, due to R.~Howe \cite{Howe1971}, as follows.  

\begin{lemma} \label{L:generic}
Any nontrivial quasi-character of a maximal torus of $\D^1$ is $\D^1$-generic.  For $T$ a maximal torus of  $\D^\times$, the quasi-character $\phi$ of $T$ of positive depth $r$ is $\D^\times$-generic if and only if $r = \min\{depth(\chi\phi) \colon \chi \in \widehat{\ratF^\times}\}$ where $\chi \phi := (\chi \circ \nrd)\otimes\phi$.
\end{lemma}

\begin{proof}
Let $\phi$ be a quasi-character of a maximal torus $T$ of $\D^\times$  of depth $r>0$ and let $S=T\cap \D^1$.  As $\mathrm{Lie}(S)$ is one-dimensional over $\ratF$, every nontrival character of $S$ is $\D^1$-generic; therefore by Proposition~\ref{P:generic}, $\phi$ is $\D^\times$-generic if and only if $\Res_S\phi$ also has depth $r$.  It thus suffices to prove that $\Res_{S_r}\phi = 1$ if and only if there exists $\chi \in  \widehat{\ratF^\times}$ such that $\Res_{T_r}\chi\phi =1$.

As $\Res_{S_r}\phi = \Res_{S_r}\chi\phi$, one direction is clear.  For the other, note first that $\LieZ = \ratF \subset \D$.  Thus $e(\LieZ_r)=1+\PF^{\lrc{r}}$, on which the norm map is the squaring map, which is bijective when $r>0$.  Thus every character of $\LieZ_r/\LieZ_{r+}$ is realized as $\chi \circ \nrd \circ e$, for some character $\chi$ of $\ratF^\times$ of depth $r$.  Choose $\chi$ such that $\chi \circ \nrd \circ e = \phi^{-1}\circ e$ on $\LieZ_r/\LieZ_{r+}$.  Since each of these characters is trivial on $\LieS_r$, it follows that they are equal on $\LieT_r$ as well, whence $\chi\phi = (\chi \circ \nrd)\phi$ is trivial on $T_r$, as required.
\end{proof}

\subsection{Depths of generic quasi-characters of tori}

Let  $\GG \in \{\D^\times,\D^1\}$.

\begin{proposition}\label{P:genericdepth}
Let $\TT$ be a maximal torus of $\GG$.
If $\TT$ is unramified then its $\GG$-generic characters have integral depth, whereas if $\TT$ is ramified then its $\GG$-generic characters have depth in $\frac12 + \mathbb{Z}$.
\end{proposition}

\begin{proof}
Each maximal torus $T$ of $\D^\times$ has the form $\ratF[\beta]^\times \subset \D^\times$, for some $\beta \in \D \setminus \ratF$; we may without loss of generality assume $\beta$ has trace $0$.   The Lie algebra of $S = T\cap \D^1$ is $\LieS = \ratF \beta$.  Thus the values $r\in \real$ for which $\LieS_r \neq \LieS_{r+}$ occur for $r\in \valD(\beta)+\mathbb{Z}$, which are integers if $T$ is unramified, and elements of $\frac12 + \mathbb{Z}$ if $T$ is ramified.  The result for $S$ and $T$ now follows from Lemma~\ref{L:generic} and Proposition~\ref{P:generic}, respectively.
\end{proof}

\begin{corollary}
Let $\TT$ be a maximal torus of $\GG$ and $\phi$ a $\GG$-generic quasi-character of $\TT$ of depth $r$.  Set $s=r/2$.   Then $\TT\GG_{s} = \TT\GG_{s+}$ unless $\TT$ is unramified and $r$ is odd.
\end{corollary}

\begin{proof}
We note that $\GG_{s}=\GG_{s+}$ unless $s \in \frac12 \mathbb{Z}$, and therefore by Proposition~\ref{P:genericdepth}, the equality follows for $\TT$ a ramified torus.  
If $\GG=\D^\times$ and $T$ is unramified then we may without loss of generality assume $T=\unextF^\times$ and decompose $\LieG = \LieT \oplus \LieT j$.  Since $\val(j)=\frac12$, for each integral $s$ we have $\D^\times_{s}/\D^\times_{s+} \cong T_s/T_{s+}$, whence $T\D^\times_{s} = T\D^\times_{s+}$.  The same argument holds for $\GG= \D^1$ and $\TT=S=T\cap \D^1$, by noting the analogous decomposition $\LieG=\LieS \oplus \LieT j$.  Finally, since $T\cap \D^\times_{s}=T_{\lrc{s}}$ it follows for $s\in \frac12 + \mathbb{Z}$ that $\D^\times_s \neq \D^\times_{s+}$ ensures $T\D^\times_s \neq T\D^\times_{s+}$, and also $S\D^1_s \neq S\D^1_{s+}$.
\end{proof}

\subsection{Smooth representations of $\D^\times$ and $\D^1$} 
The smooth irreducible representations of $\D^\times$ and $\D^1$ are well-known (and are evidently all supercuspidal), see \cite{Corwin1974, Howe1971, BushnellHenniart2006} and \cite{Misaghian2005} respectively.  We present the complete list for the case of $p\neq 2$ (the \emph{tame} case) here; the case $p=2$ for includes more representations and for $\D^\times$ is treated in, for example, \cite[Ch 13]{BushnellHenniart2006}.
For simplicity, we reserve ``representation of depth $\star$'' for the subset of those of degree greater than one (that is, excluding the quasi-characters).

\subsubsection{Characters}
Since $\D^\times/[\D^\times,\D^\times] \cong \D^\times/\D^1 \cong \ratF^\times$  the one-dimensional smooth representations of $\D^\times$ are in bijection with characters of $\ratF^\times$ via the nrd map. 
On the other hand, as $[\D^1,\D^1] = \D^1 \cap (1+\PD)=\D^1_{0+}$, every character of $\D^1$ factors through $\D^1/\D^1_{0+} \cong \unextF^1/\unextF^1_{0+}$, so they are in bijection with the $q+1$ distinct depth-zero characters of $\unextF^1$.

\subsubsection{Depth-zero representations}
Since $\D^1=\D^1_0$, the depth-zero representations of $\D^1$ are those which factor through $\D^1/\D^1_{0+}$, namely its characters, so by our convention we will say $\D^1$ has no representations of depth zero.  

In contrast $\D^\times$ admits depth-zero representations, whose construction we summarize from \cite[\S 54.2]{BushnellHenniart2006} as follows.
A depth-zero generic or \emph{admissible} character of $\unextF^\times$ (see \cite{Howe1971}) is a quasi-character $\theta$ of $\unextF^\times$ of depth zero which does not factor through the norm map, or equivalently, such that $\theta \neq \overline{\theta}$ where $\overline{\theta}(z)=\theta(\overline{z})$.  Two distinct admissible characters $\theta'$ and $\theta$ are called $\ratF$-equivalent if $\theta'=\overline{\theta}$.

Given a depth-zero admissible character $\theta$ of $T=\unextF^\times$, extend it trivially across $\D_{0+}$ to give a quasi-character $\theta$ of $T\D^\times_{0+}$.  Then
$$
\pi(\theta) = \Ind_{T\D^\times_{0+}}^{\D^\times} \theta
$$
is an irreducible representation of $\D^\times$ of depth zero.  Moreover, isomorphism  classes of depth-zero representations of $\D^\times$ are in bijection with $\ratF$-equivalence classes of depth-zero admissible characters of $\unextF^\times$.

\begin{remark}\label{two}
Since $T\D^\times_{0+} = \unextF^\times (1+\PD) = \unextF^\times \D^1$, we see $\D^\times/T\D^\times_{0+}$ is represented by $\{1,\sigma=j\}$.  Thus $\pi(\theta)$ has degree $2$ and its restriction to $T\D^\times_{0+}$ is exactly $\theta \oplus \overline{\theta}$.
\end{remark}

\subsubsection{Positive-depth representations}

Let $\GG \in \{\D^\times, \D^1\}$.  There are two kinds of representations of positive depth of $\GG$.  The first are those of the form $\pi = \chi\pi_0$, where $\pi_0$ is any representation of depth zero and $\chi$ is any character of positive depth.  

The second kind are parametrized by generic toral $\GG$-data $\Psi=(\TT,y,\phi,r,\chi)$ of length one, as in Section~\ref{S:Yu}.  For simplicity, we omit $y$, as it is the unique point $x$ in $\buil^{red}(\G,\ratF)$.  Similarly, we may omit $\chi$ when $\GG = \D^1$, because all characters of $\D^1$ are of depth zero and hence are subsumed in $\phi$.  As always, we identify a character $\chi$ of $\D^\times$ with a character of $\ratF^\times$ via $\nrd$.

\section{Relating representations of $\D^\times$, $\D^1$ and $\ODt$} \label{S:branching}

In this section we apply the results of Section~\ref{S:derived} to determine the restrictions and decomposition into irreducible representations of each of the representations of $\D^\times$ to $\OD^\times$ and to $\D^1$.  The restriction to $\D^1$ has presumably been known to experts.  The restriction to $\ODt$ is new and leads to the datum-type classification of positive-depth representations of $\ODt$ in Section~\ref{SS:class}.

\subsection{Branching rules for the restriction of representations of $\D^\times$ to $\D^1$} \label{S:dd1}

\begin{lemma}
The restriction of any character of $\D^\times$ to $\D^1$ is trivial.  All nontrivial characters of $\D^1$ occur in the restriction of a depth-zero representation of $\D^\times$.
\end{lemma}

\begin{proof}
The first statement follows from $\D^1 = [\D^\times,\D^\times]$.  Let $\theta$ be a depth-zero admissible character of $T=\unextF^\times$ and $\pi(\theta)$ be the associated depth-zero representation of $\D^\times$.  Set $\vartheta = \Res_{\D^1\cap \unextF^\times}\theta$; then by Remark~\ref{two} $\Res_{\D^1}\pi(\theta)=\vartheta \oplus \overline{\vartheta}$.   As $\D^1\cap \unextF=\unextF^1 = \{z\overline{z}^{-1} \mid z \in \unextF^\times\}$, the admissibility of $\theta$ is equivalent to $\vartheta \neq 1$.  
\end{proof}

In particular, the restriction of $\pi(\theta)$ to $\D^1$ decomposes with multiplicity one \emph{except} when $\vartheta^2=1$, when the multiplicity is two. 


For a $\D^1$-datum $\Psi^1 = (S,\phi^1,r)$ set $\overline{\Psi^1} = \lconj{\sigma}{\Psi^1}=(S,\overline{\phi^1},r)$.

\begin{proposition} \label{P:restD1}
Let $\Psi$ be a $\D^\times$-datum and let $\Psi^1$ denote its restriction to $\D^1$.  Then $\Res_{\D^1}\pi(\Psi)$ decomposes as a direct sum of two inequivalent representations.  When $\Psi$ is unramified, or when $-1 \in {\ratF^\times}^2$, we have
$$
\Res_{\D^1}\pi(\Psi) \cong \pi(\Psi^1) \oplus \pi(\overline{\Psi^1})
$$ 
whereas otherwise 
$$
\Res_{\D^1}\pi(\Psi) \cong \pi(\Psi^1) \oplus \pi(\lconj{\mu}{\Psi^1}).
$$ 
\end{proposition}

\begin{proof}
It suffices by  Theorem~\ref{T:decomp} to note that $\D^\times/T\D^1$ is represented by $\{1,\sigma\}$ except in the case that $T$ is ramified and $-1 \notin {\ratF^\times}^2$, where it is represented by $\{1,\mu\}$.  
\end{proof}

\subsection{Branching rules for the restriction of representations of $\D^\times$ to $\ODt$}

Note that the center of $\ODt$ is $\OFt$.  Since $\OFt \D^1$ has index two in $\ODt$, and $\ratF^\times \ODt$ has index two in $\D^\times$,  each restriction, from $\D^\times$ to $\ODt$, or from $\ODt$ to $\D^1$, is either irreducible or else a direct sum of two inequivalent irreducible representations.
 We may thus deduce many of the branching rules for $\ODt$ from the results of the preceding section. 
We begin with the characters.

\begin{lemma} \label{L:char}
Each character of $\ODt$ may be uniquely written as $\chi\theta := (\chi \circ \nrd)\theta$, 
with $\chi \in \widehat{\OF^\times}$ and $\theta$ either trivial, or else the inflation of an admissible depth zero character of $\OL^\times$ to $\ODt$.
\end{lemma}

\begin{proof}
Each character $\varphi$ of $\ODt$ occurs in the restriction of some representation $\pi$ of $\D^\times$. 
Its further restriction to $\D^1$ being a character implies by the preceding section that $\pi = \chi\pi(\theta)$ for some $\chi\in\widehat{\ratF^\times}$ and admissible character $\theta$ of the unramified torus $\unextF^\times$.  Set $\chi_0 = \Res_{\OF^\times}\chi$ and identify $\theta$ with a depth-zero character of $\ODt$ via $\ODt/(1+\PD) \cong \unextF^\times_0/\unextF^\times_{0+}$.  Then by Remark~\ref{two} $\Res_{\ODt}\chi\pi(\theta) = \chi_0\theta \oplus \chi_0\overline{\theta}$.  The unicity is immediate.
\end{proof}

We now turn to the restrictions of representations of positive depth of $\D^\times$ (of the second kind).

\begin{proposition} \label{P:OD}
Let $\Psi= (T,\phi,r,\chi)$ be a $\D^\times$-datum.  Then if $T$ is unramified
\begin{equation}\label{E:twopiece}
\Res_{\OD^\times}\pi(\Psi) \cong  \Ind_{T_0\D^\times_s}^{\ODt}\rho(\Psi) \oplus \Ind_{T_0\D^\times_s}^{\ODt} \rho(\overline{\Psi})
\end{equation}
is a decomposition into irreducible inequivalent representations of $\ODt$ whereas if $T$ is ramified, then
\begin{equation}\label{E:onepiece}
\Res_{\OD^\times}\pi(\Psi) \cong \Ind_{T_0\D^\times_s}^{\ODt} \rho(\Psi)
\end{equation}
is irreducible.
\end{proposition}

\begin{proof}
If $T$ is unramified then $\OD^\times\backslash \D^\times / T \D^\times_s = \{1,\sigma\}$  so by Mackey theory 
$$
\Res_{\OD^\times}\pi(\Psi) \cong \Ind_{T_0\D^\times_s}^{\ODt}\rho(\Psi) \oplus \Ind_{T_0\D^\times_s}^{\ODt}\lconj{\sigma}{\rho(\Psi)}
$$
where we have used that $\OD^\times \cap TG_{y,s}=T_0\D^\times_s$ and that $\sigma$ normalizes this group.  Since there are two factors, they must be irreducible.  Note that
$$
\Res_{\D^1}\Ind_{T_0\D^\times_s}^{\ODt}\rho(\Psi) \cong \pi_{\D^1}(\Psi^1)
$$
so their inequivalence follows from Proposition~\ref{P:restD1}, for example.  Applying Lemma~\ref{L:mackeygen} yields \eqref{E:twopiece}.

If $T$ is ramified then $\D^\times=\ODt T$, whence \eqref{E:onepiece}.  By Proposition~\ref{P:restD1} its further restriction to $\D^1$ decomposes as a sum of two invariant subspaces $\pi_{\D^1}(\Psi^1) \oplus \pi_{\D^1}(\lconj{\gamma}{\Psi^1})$, where $\gamma = \sigma$ if $-1\in (\ratF^\times)^2$ and $\gamma=\mu$ otherwise.  In either case, $\gamma \in \ODt$, whence it follows from Mackey theory that neither subspace can be invariant under $\ODt$.  Thus $\Res_{\ODt}\pi(\Psi)$ is irreducible.
\end{proof}

It follows from the proof that for $\Psi=(T,\phi,r,\chi)$, $\Res_{\D^1}\Ind_{T_0\D^\times_s}^{\ODt}\rho(\Psi)$ is irreducible if and only if $T$ is unramified.

\subsection{Classification of irreducible representations of $\ODt$} \label{SS:class}

Recall the equivalence relation on $G$-data defined by Proposition~\ref{P:HM}.  In this section we provide its analogue for the group $\ODt$.

Say that two $\D^\times$-data $\Psi=(T,\phi,r,\chi)$ and $\Psi'=(T',\phi',r,\chi')$ are \emph{$\ODt$-equivalent}, written $\Psi \eqo \Psi'$, if there exists a $g\in \ODt$ for which $\lconj{g}{T}=T'$, $r=r'$ and $\Res_{T_0'}\lconj{g}{(\chi\phi)} = \Res_{T_0'}\chi'\phi'$.  
Let $\Psi_0=(T_0,\phi_0,r,\chi_0)$ where $\phi_0$ and $\chi_0$ are the corresponding characters restricted to $T_0$ and $\OF^\times$, respectively.  We call $\Psi_0$ an \emph{$\ODt$-datum} and say that $\Psi_0$ and $\Psi_0'$ are \emph{equivalent} if there exists $g\in \ODt$ such that $\lconj{g}{T_0}=T_0'$, $r=r'$ and $\lconj{g}{(\chi_0\phi_0)}=\chi_0'\phi_0'$.  
Note that for $g\in \D^\times_0=\ODt$, $\lconj{g}{T}=T'$ is equivalent to $\lconj{g}{T_0}=T_0'$.  
It follows that $\ODt$-equivalence classes of $\D^\times$-data are in bijection with equivalence classes of $\ODt$-data.

Given a $\D^\times$-datum $\Psi$, let 
$$
\pi_{\ODt}(\Psi) = \Ind_{T_0G_s}^{\ODt} \rho(\Psi),
$$
which is irreducible by Proposition \ref{P:OD}.  The following theorem implies $\pi_{\ODt}(\Psi)$ depends only on the equivalence class of $\Psi_0$, in analogy with Proposition~\ref{P:HM}, whence a datum-type classification of the representations of positive depth of $\ODt$.

\begin{theorem} \label{T:classificationOD}
The irreducible representations of $\OD^\times$ are:
\begin{enumerate}
\item the distinct characters:  $\chi\theta := (\chi\circ\nrd)\theta$, where $\chi \in \widehat{\OF^\times}$ and $\theta$ is either trivial or the inflation to $\OD^\times$ of an admissible depth-zero character of $\unextF^\times \subset \OD^\times$; 
\item the representations of degree greater than one: $\pi_{\ODt}(\Psi)$, for a $\D^\times$-datum $\Psi = (T,\phi,r,\chi)$.
\end{enumerate}
Moreover, $\pi_{\ODt}(\Psi) \cong \pi_{\ODt}(\Psi')$ if and only if $\Psi \eqo \Psi'$. 
\end{theorem}

\begin{proof}
 The first point is Lemma~\ref{L:char}.  That the list in the second point is exhaustive follows from the classification of representations of $\D^\times$ and Proposition~\ref{P:OD}.  
We have only to prove the last statement. 

First suppose $\Psi \eqo \Psi'$. 
Since $\rho(\lconj{g}{\Psi})\cong \lconj{g}{\rho(\Psi)}$ for each $g\in\D^\times$, it follows easily that for any $g\in \ODt$, $\pi_{\ODt}(\lconj{g}{\Psi}) \cong \pi_{\ODt}(\Psi)$.  Therefore without loss of generality we may replace $\Psi'$ with an $\ODt$-conjugate of the form $(T,\phi',r,\chi')$ such that 
$\Res_{T_0}\chi\phi = \Res_{T_0}\chi'\phi'$.

Set $\varphi = \chi'\phi'(\chi\phi)^{-1}$; since this is a character of $T$ trivial on $T_0$, we deduce that $\Psi'':= (T,\varphi \phi,r,\chi)$ is also a $\D^\times$-datum.  Since $\chi(\varphi \phi) = \chi'\phi'$,  Proposition~\ref{P:HM} implies $\rho(\Psi'') \cong \rho(\Psi')$, whence their restrictions to $T_0G_s$ are equivalent.  On the other hand, since $\Res_{T_0}\varphi \phi = \Res_{T_0}\phi$ and $\Psi''$ and $\Psi$ share the same $\D^\times$-character $\chi$, 
it follows from Lemma~\ref{L:equiv1} that $\Res_{T_0G_s}\rho(\Psi)=\Res_{T_0G_s}\rho(\Psi'')$.  Consequently $\pi_{\ODt}(\Psi) \cong \pi_{\ODt}(\Psi')$, as required.

Now suppose $\pi_{\ODt}(\Psi) \cong \pi_{\ODt}(\Psi')$.  
Let $\Psi^1$ and ${\Psi'}^1$ denote the restrictions of $\Psi$ 
and $\Psi'$ to $\D^1$, respectively.

By the proof of Proposition~\ref{P:OD}, if $\Res_{\D^1}\pi_{\ODt}(\Psi)$ is irreducible then $\pi_{\D^1}(\Psi^1) \cong \pi_{\D^1}({\Psi'}^1)$, whereas if it is reducible then
$$
\pi_{\D^1}(\Psi^1) \in \{\pi_{\D^1}({\Psi'}^1),  \pi_{\D^1}(\lconj{\mu}{\Psi'}^1)\}.
$$
Since $\mu\in \ODt$, we may replace $\Psi'$ by an $\ODt$ conjugate if necessary to assume $\pi_{\D^1}(\Psi^1) \cong \pi_{\D^1}({\Psi'}^1)$ in this case as well.

Thus we may replace $\Psi'$ by a $\D^1\subseteq \ODt$ conjugate to assume that $S=T\cap \D^1=T'\cap \D^1$, $r=r'$ and $\phi^1 = {\phi'}^1$.  In terms of $\D^\times$-data, it follows that $T=T'$ and $\Res_{S}\phi = \Res_S \phi'$.  On the other hand, comparing central characters and using Lemma~\ref{L:equiv1} yields $\Res_{Z_0}\chi\phi = \Res_{Z_0}(\chi'\phi')$.  Thus $\chi\phi$ and $\chi'\phi'$ agree on $Z_0S$.   When $T=\extF^\times$ is ramified, then $Z_0S=\OF^\times\extF^1 = \extF^\times\cap\ODt=T_0$ so we may conclude $\Psi \eqo \Psi'$.

When $T$ is unramified, then $Z_0S$ is of index two in $T_0$. Choose a character $\xi$ of $T$ which restricts on $T_0$ to the nontrivial character of $T_0/Z_0S$; then $\Res_{T_0}\chi'\phi' \in \{\Res_{T_0}\chi\phi, \Res_{T_0}\xi\chi\phi\}$.  

Suppose for the purpose of contradiction that $\Res_{T_0}\chi'\phi'=\Res_{T_0}\xi\chi\phi$. 
Then $\Psi'\eqo \Psi''$ where 
 $\Psi'' = (T,\xi\phi,r,\chi)$.  Since $\xi$ is of depth zero Lemma~\ref{L:equiv1} implies that the pullbacks $\omega, \omega_\xi$ of the Heisenberg-Weil lifts corresponding to $\phi$ and $\xi\phi$, respectively, coincide. Since $\xi$ is trivial on $G_s$, $\rhop'' = \xi \rhop$ 
whence $\rho(\Psi'')=\xi\rho(\Psi)$.  
To derive a contradiction it suffices by Mackey theory to show that for all $\gamma \in \ODt$, 
\begin{equation} \label{E:cant}
H_\gamma := \Hom_{T_0G_s \cap \lconj{\gamma}{T_0}G_s}(\xi\rho(\Psi),\lconj{\gamma}{\rho(\Psi)}) = \{0\}.
\end{equation}
This is true for $\gamma \in T_0G_s$.  For $\gamma \notin T_0G_s$, we may without loss of generality assume that $T=\unextF^\times$ and by scaling by an element of $T_0 = \OL^\times$, that $\gamma = 1+zj$, with $z\in \OL$, $\valD(zj)=m < s$.  For any $\delta \in \OL^\times$, we have
$$
\lconj{\gamma}{\delta} = \nrd(\gamma)^{-1}\left( (\delta-\varpi z\overline{z}\overline{\delta})+ (\delta-\overline{\delta})zj \right),
$$
which lies in $T_0G_s$ if and only if $\delta \in Z_0T_{s-m} \subseteq Z_0T_{0+}$.
Thus 
$T_0G_s \cap \lconj{\gamma}{T_0}G_s \subseteq Z_0T_{0+}G_s$, on which $\rho$ and $\xi\rho$ agree.  
Thus for $\gamma \notin T_0G_s$, we have
$$
H_\gamma = \Hom_{T_0G_s \cap \lconj{\gamma}{T_0}G_s}(\rho(\Psi),\lconj{\gamma}{\rho(\Psi)})=\{0\},
$$
since the irreducibility of $\pi_{\ODt}(\Psi)$ implies that only $\gamma \in T_0G_s$ can support an intertwining operator.   Consequently \eqref{E:cant} holds for all $\gamma\in\ODt$, our contradiction.
\end{proof}

We deduce that the representations of degree greater than one are parametrized by $\ODt$-conjugacy classes of $\ODt$-data $\Psi_0 = (T_0,\phi_0,r,\chi_0)$ where these each represent the restriction of a $\D^\times$-datum to $\ODt$.

\section{Remarks on the matching of types}\label{S:matching}

The Jacquet-Langlands correspondence asserts a bijection between the irreducible representations of $\D^\times$ and the irreducible square-integrable representations of $\GL(2,\ratF)$, characterized by a matching of $L$-functions and $\varepsilon$-factors, or simply by a matching of characters on the regular elliptic sets of the two groups (which are in natural correspondence) \cite[\S 56]{BushnellHenniart2006}.   The representations of $\GL(2,\ratF)$ which occur are (up to twisting by characters of $\GL(2,\ratF)$), precisely: (a) the supercuspidal representations, which are determined by characters of tori, and (b) the Steinberg representation.  

For $p\neq 2$, the correspondence is simply stated; see \cite[\S 3]{Moy1986} or \cite[\S 56]{BushnellHenniart2006}.  The Steinberg representation of $\GL(2,\ratF)$ corresponds to the trivial representation of $\D^\times$.  Each torus of $\GL(2,\ratF)$ or of $\D^\times$ corresponds to a quadratic field extension of $\ratF$, up to conjugacy, so they are in natural correspondence and we abusively use the same letter $T$ to denote corresponding tori.   To an unramified torus $T$ and a generic character $\theta$ of depth zero, one associates the depth-zero supercuspidal representation of $\GL(2,\ratF)$ obtained by inflating the Deligne-Lusztig cuspidal representation $R_T(\theta)$ to $\GL(2,\OF)$ and compactly inducing this to $\GL(2,\ratF)$.  Through the Jacquet-Langlands correspondence it is identified with $\pi(\theta)$.  On the other hand, to any torus $T$ and character $\phi$ of positive depth $r$, there is a small correction factor: one associates $\pi_{\GL(2,\ratF)}(T,y,\phi,r,\chi)$ with $\pi_{\D^\times}(T, \phi\eta, r, \chi)$, where in each case $\chi \in \widehat{\ratF^\times}$, and $\eta$ is the quadratic unramified character of $\D^\times$ defined by $\eta(x) = (-1)^{2\val_D(x)}$. 

It is known that the Jacquet-Langlands correspondence does not descend to a correspondence of representations of the derived groups.  For example, $SL(2,\ratF)$ admits an $L$-packet with four elements whereas we saw in Section~\ref{S:dd1} that all representations of $\D^1$ occur in packets of size one or two.  It is also unreasonable to expect the correspondence to descend to one of the associated  maximal compact open subgroups, since, for example, the supercuspidal representations of $\GL(2,\ratF)$ decompose into infinitely many components upon restriction to $\GL(2,\OF)$, some of which are common to all representations of the same central character (as may be deduced from \cite{Nevins2013}).

On the other hand, it is expected \cite[\S 56]{BushnellHenniart2006} that the correspondence preserves types.  This is trivial for the trivial-Steinberg pair, and follows directly in the depth-zero case from properties of Deligne-Lusztig cuspidal representations.  In the positive-depth cases, each of $\GL(2,\ratF)$ and $\D^\times$ contain a unique maximal compact open subgroup $K$ up to conjugacy, and each type can be realized as the inducing datum for an irreducible representation of least depth occuring in the restriction of a representation to $K$.  In the case of $\GL(2,\ratF)$, the restriction to $\GL(2,\OF)$ has a unique component of minimal depth; in the case of $\D^\times$, there are two if the torus is unramified.  In this latter case, we have shown the inducing data are $\D^\times$-conjugate.  Thus the correspondence is well-defined.

\section*{Acknowledgments}  The author conducted this research during a wonderful visiting year at the \emph{Institut de Math\'ematiques et de Mod\'elisation de Montpellier}, Universit\'e Montpellier II, at the invitation of Ioan Badulescu.  My thanks are also due to Jeff Adler, who suggested several improvements to a first draft of this article, in particular allowing some unnecessary hypotheses to be removed.  

\bibliography{padicrefs}{}
\bibliographystyle{amsplain}

\end{document}